\documentclass[10pt]{article}

\usepackage{amssymb,amsmath,amsopn,amsthm}
\usepackage[mathscr]{euscript}
\usepackage[margin=1.5in]{geometry}
\usepackage{graphicx}
\usepackage{float}
\usepackage{tikz-cd}
\usepackage[utf8]{inputenc}
\usepackage{hyperref}

\newtheorem{lemma}{Lemma}
\theoremstyle{definition}
\newtheorem{defn}{Definition}
\newtheorem*{thm:Comm}{Theorem \ref{thm:Comm}}
\newtheorem*{thm:C_infty}{Theorem \ref{thm:C_infty}}
\newtheorem{Theorem}{Theorem}

\theoremstyle{remark}

\title{Topological perspective on Statistical Quantities II}
\date{}
\author{Nissim Ranade}

\begin{document}
\maketitle

\section{Introcuction}

In the first part of this paper \cite{Top.persp} we discuss the relationship between the Boolean cumulants and $A_\infty$ morphisms. In this paper we will will extend that idea to the regular cumulants and $C_\infty$ morphisms. Cumulants can be defined in general for linear maps between commutative algebras which do not respect the algebraic structure. 

The first few cumulants for a map $e$ between two commutative algebras are defined as follows.
\[k_1(a) = e(a) \]
\[k_2(a,b) = e(ab)-e(a)e(b) \]
\[k_3(a,b,c) = e(abc)- e(ab)e(c) - e(a)e(bc) - e(ca)e(b) + 2e(a)e(b)e(c)\]

In general, $k_n$ is defined as follows. 

\[k_n(a_1,a_2,\ldots,a_n) = \sum_\pi(|\pi|-1)!(-1)^{|\pi|-1}\prod_{b\in \pi}e(\prod_{i\in b}a_i)  \]The sum is taken over all partitions of $\{1,\ldots,n\}$. Knowing the cumulants allows you to calculate the expectations of products. For example 
\[e(ab) = k_2(a,b)+k_1(a)k_1(b)\] 
\[e(abc) = k_3(a,b,c)+k_2(a,b)k_1(c)+k_2(b,c)k_1(a)+k_2(c,a)k_1(b)+k_1(a)k_1(b)k_1(c)\]

More generally cumulants can be defined using the above formulas for chain maps between differential graded commutative algebras. Furthermore for linear maps between $C_\infty$ algebras, cumulants can be defined up to homotopy.

Consider the differential forms $\Omega(M)$ on a manifold which form a differential graded commutative algebra. The commutative associative product can be transferred to a $C_\infty$ structure on the cochains $C(M)$ on the manilfold via a map $I$ that integrates each form appropriately on the chains. The map $I$ is the first term of a $C_\infty$ morphism between the two structures and the cumulants of $I$ can be defined up to homotopy. We have the following theorem that generalizes the theorem from the previous paper to $C_\infty$ morphisms.

\begin{thm:Comm}

Let $(A,d_A)$ and $(B,d_B)$ be two \textit{dgca}s. Let $p$ be a chain map from $A$ to $B$. Let $k_2$, $k_3$ and so on be the cumulants of $p$. Suppose $p$ is the first term of a $C_\infty$ morphism $(p, p_2, p_3, \ldots)$ where $p_n:A^{\otimes n}\rightarrow B$. Then the following statements hold.
\begin{itemize}
\item[i)] $p_2$ gives a homotopy from the second commutative cumulant $k_2$ to zero. All the higher cumulants $k_n$ are also homotopic to zero using maps created by $p_2$ and $p_1$.
\item[ii)] $p_3$ gives a homotopy between different ways of making $K_3$ homotopic to zero. For all the higher Boolean cumulants, homotopies between the multiple different ways of making them homotopic to zero are homotopic to each other using $p_3$, $p_2$ and $p_1$.
\item[iii)] In general any cycles that are created using the homotopies $\{p_j\}_{j=1}^n$ are made homotopic to zero using maps made by $\{p_j\}_{j=1}^{j+1}$.
\end{itemize}
\end{thm:Comm}

Much like in the case of an $A_\infty$ morphism between associative algebras, a $C_\infty$ morphism between \textit{dgca}s makes the cumulants completely collapse. We can generalize the above theorem to $C_\infty$ morphisms between $C_\infty$ algebras.

\begin{thm:C_infty}
Let $A$ and $B$ be two $C_\infty$ algebras. Let $p$ be a chain map from $A$ to $B$. Let $k_2$, $k_3$ and so on be the regular cumulants of $p$ defined up to homotopy. Suppose $p$ is the first term of an $C_\infty$ morphism $(p, p_2, p_3, \ldots)$ where $p_n:A^{\otimes n}\rightarrow B$. Then the following statements hold.
\begin{itemize}
\item[i)] $p_2$ gives a homotopy from the second cumulant $k_2$ to zero. All the different ways of defining the higher cumulants $k_n$ are also homotopic to zero using maps created by $p_2$ and $p_1$.
\item[ii)] $p_3$ gives a homotopy between different ways of making $k_3$ homotopic to zero. For all the higher cumulants, homotopies between the multiple different ways of making them homotopic to zero are homotopic to each other using $p_3$, $p_2$ and $p_1$.
\item[iii)] In general any cycles that are created using the homotopies $\{p_j\}_{j=1}^n$ are made homotopic to zero by maps made using $\{p_j\}_{j=1}^{n+1}$.
\end{itemize}

\end{thm:C_infty}

\section{$C_\infty$ Algebras and their morphisms}

\begin{defn}
An \textit{differential graded commutative algebra} or a \textit{dgca} is a differential graded algebra $(A,d,m)$ such that $m$ is graded commutative. That is \[m(a\otimes b) = (-1)^{|a||b|}m(b\otimes a)\]
\end{defn}

The differential forms on a manifold are an example of a \textit{dgca}. Just like $A_\infty$ algebras generalize the notion of differential graded associative algebras, the notion of a \textit{dgca} is generalized by $C_\infty$ algebras. A $C_\infty$ algebra is an $A_\infty$ algebra such that the maps $m_n$ satisfy certain equations involving $(q,r)$-shuffles, where $q+r=n$. 

\begin{defn}
A \textit{$(q,r)$-shuffle} is a permutation $\sigma$ of $(1,2,\ldots,q+r)$ such that
\begin{itemize}
\item if $1\leq i \leq j\leq q$, then $\sigma(i)\leq\sigma(j)$
\item if $q+1\leq i \leq j\leq q+r$, then $\sigma(i)\leq\sigma(j)$
\end{itemize} 
\end{defn}

\begin{defn}
A \textit{$C_\infty$ algebra} is an $A_\infty$ algebra $(A,m_1, m_2, \ldots)$ such that for every ordered pair of positive integers $(q,r)$ and $(x_1, x_2,\ldots,x_{q+r})$ where $x_i\in A$  
\[ m_{q+r}(\sum_{\sigma \in (q,r)-\text{shuffles}} (x_{\sigma^{-1}(1)},x_{\sigma^{-1}(2)},\ldots,x_{\sigma^{-1}(q+r)}))=0 \]
\end{defn}

\begin{defn}
A \textit{$C_\infty$ morphism} is an $A_\infty$ morphism $P=(p_1,p_2,\ldots)$ such that for every ordered pair of positive integers $(q,r)$ and $(x_1, x_2,\ldots,x_{q+r})$ where $x_i\in A$  
\[ p_{q+r}(\sum_{\sigma \in (q,r)-\text{shuffles}} (x_{\sigma^{-1}(1)},x_{\sigma^{-1}(2)},\ldots,a_{\sigma^{-1}(q+r)}))=0 \]
\end{defn}

\section{Transferring structure from differential forms to cochains}

Suppose $\Omega^*(M)$ is the algebra of differential forms on a smooth manifold $M$ and suppose $C^*(M)$ is the cochain complex corresponding to a certain fixed regular cell decomposition of $M$. This complex is the space of real valued functions on the chains of the regular cell decomposition. Consider the map $I$ defined as follows. For a differential form $\omega$ and a cell $\sigma$ of the cell decomposition
\[ I(\omega)(\sigma) = \int_{\sigma}\omega \] 

The cochains $C^*(M)$ have a canonical basis given by the cells of the decomposition. A basis element $\sigma^*$ corresponding to a cell $\sigma$ is a map such that
\[\sigma^*(\sigma) = 1\]
\[\sigma^*(\tau ) = 0\]
for every cell $\tau\neq\sigma$.

A map $i$ from $C^*(M)$ to $\Omega^*(M)$ can be constructed to have the following properties.
\begin{itemize}
\item[1)] $i(\sigma^*)$ integrates to $1$ on $\sigma$.
\item[2)] $i(\sigma^*)$ is supported only in a small neighborhood of the interior of $\sigma$. 
\item[3)] $i(\sigma^*)$ integrates to zero on all cells that have the same dimension as $\sigma$ but are not $\sigma$.
\end{itemize}

The map $i$ is an inclusion of the cochains into differential forms. The map $I\circ i$ is equal to identity on $C^*(M)$ and $ i\circ I$ is homotopic to identity on $\Omega^*(M)$. The homotopy $h$ is a map of degree $-1$ on $\Omega^*(M)$ such that \[dh+hd = i\circ I - id\] $h$ can be constructed inductively on cells and then glued together on the whole manifold. Given a choice of $i$ and $h$ the wedge product on the forms can be transferred to give an $A_\infty$ structure on $C^*(M)$.\cite{konstevich-soibelman}\cite{Merkulov} Further since the differential forms are graded commutative, we can pick the transferred structure to be $C_\infty$. \cite{homotopy-comm-transfer}

Now suppose $D^*(M)$ is another cochain complex corresponding to a finer cell decomposition of the original cell decomposition. Every cell of the original cell decomposition can be written as a union of cells of the finer cell decomposition. Thus there is a map $p$ from $D^*(M)$ to $C^*(M)$. There are projections $p_D$ and $p_C$ from the differential forms to $D^*(M)$ and $C^*(M)$ respectively given by integrating the forms of the cells of the complexes. 

\begin{equation}
\begin{tikzcd}
 {} & D^*(M) \arrow{d}{p}  \\
  \Omega^*(M) \arrow{r}{p_C}\arrow[dashed]{ur}{p_D} & C^*(M)
\end{tikzcd}
\end{equation}

We can pick inclusions (right inverses) of $D^*(M)$ and $C*(M)$ into the differential forms. However, the transfer maps for the transferred structure might not necessarily commute. For the maps to commute it is necessary to pick the inclusions and the homotopies appropriately. We first transfer the multiplicative structure of $\Omega^*(M)$ to an $A_\infty$ structure on $D^*(M)$. For this we pick an inclusion $i_D$ and a homotopy $h_D$ such that $dh_D+h_Dd = i_D\circ p_D - id $

\[
\begin{tikzcd}
\arrow[loop left]{l}{h_D} \arrow[r,"p_D"] \Omega^*(M)  & \arrow[l,bend left,"i_D"] D^*(M)
\end{tikzcd}
\]

Since the map $p$ is a quasi-isomorphism and since we are working over field coefficients there is an inverse quasi-isomorphism $i$ up to homotopy. Also since $p$ is a projection $i$ can be picked to be an inclusion such that $p\circ i$ is identity on $C^*(M)$. For this inclusion we can pick a homotopy $h$ on $D^*(M)$ such that $dh+hd = i\circ p$. We can transfer the structure from $D^*(M)$ to $C^*(M)$ using $i$ and $h$. 

\[
\begin{tikzcd}
\arrow[loop left]{l}{h} \arrow[r,"p"] D^*(M)  & \arrow[l,bend left,"i"] C^*(M)
\end{tikzcd}
\]

Consider the map $i_C = i_D\circ i$ which is an inclusion of $C^*(M)$ into the differential forms $\Omega$.

\begin{lemma}
$h_C = i_D\circ h\circ p_D + h_D$ is a homotopy from $i\circ p$ to identity. That is $dh_C+h_Cd = i\circ p+id$.
\end{lemma}

\begin{proof}
\[dh_C+h_Cd = d(i_D\circ h\circ p_D) + (i_D\circ h\circ p_D)d + dh_D+h_Dd \] Since $p_D$ and $i_D$ are chain maps this is equal to 
\[i_D\circ dh\circ p_D + i_D\circ hd\circ p_D + dh_D+h_Dd \] Since $h$ and $h_D$ are homotopies this is equal to
\[i_D\circ i\circ p\circ p_D - i_D\circ p_D + i_D\circ p_D - id \]
\[=i_C\circ p_C - id\]
\end{proof}

Thus we can transfer the structure from the differential forms directly to $C^*(M)$. 

\[
\begin{tikzcd}
\arrow[loop left]{l}{h_C} \arrow[r,"p_C"] \Omega^*(M)  & \arrow[l,bend left,"i_C"] C^*(M)
\end{tikzcd}
\]

\begin{lemma}
The $A_\infty$ structure on $C^*(M)$ that is transferred from $\Omega^*(M)$ is the same as the one transferred from $D^*(M)$. 
\end{lemma}
\begin{proof}
Recall that the formula for the transferred structure using $i$ and $h$ is as follows. 

\begin{figure}[H]
\centering \includegraphics[width=5in]{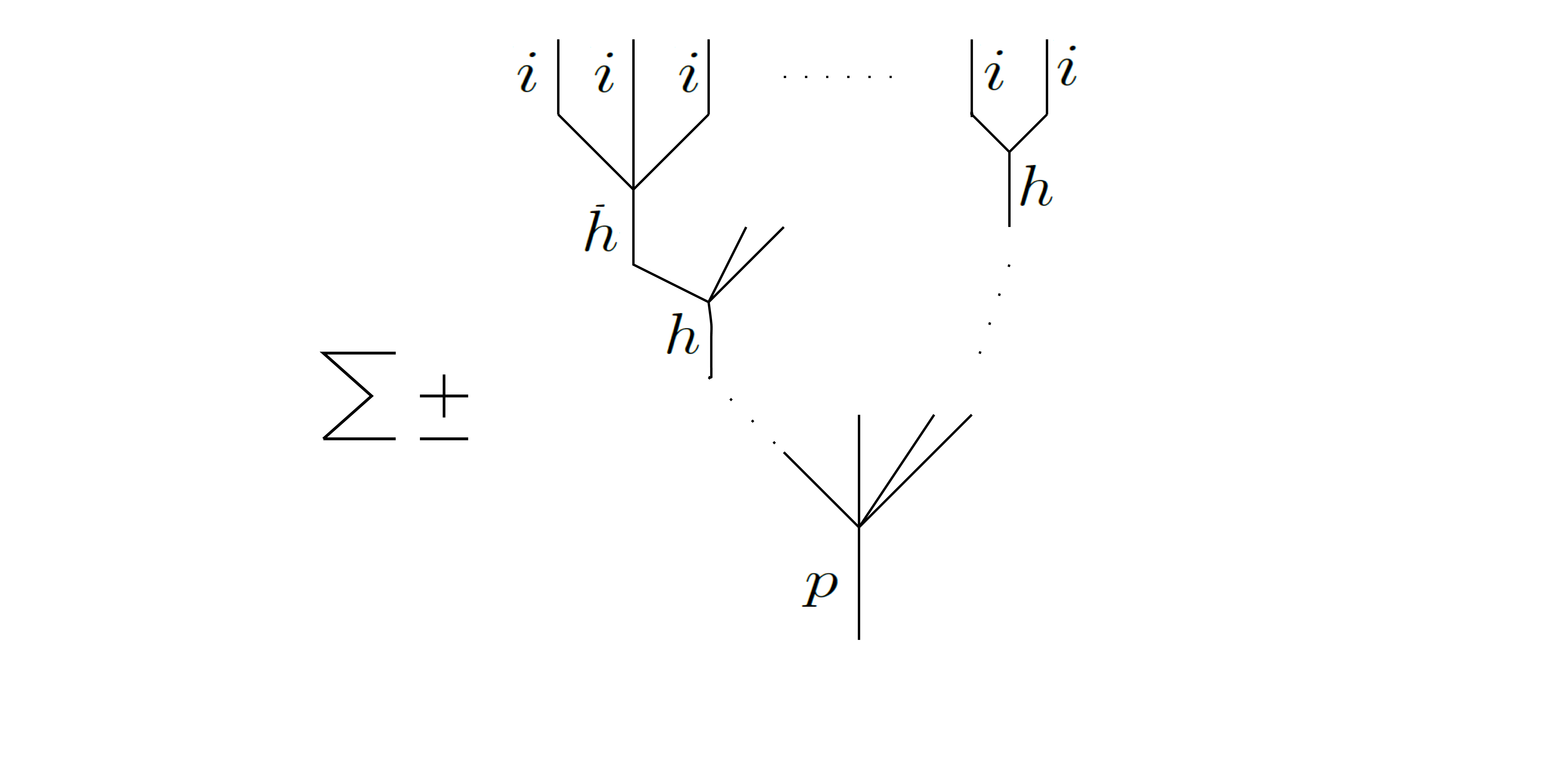}
\caption{\label{m_n}}
\end{figure}

In the above diagram the nodes of the trees correspond to the maps $m_n$ in the structure transferred on $D^*(M)$ from the differential forms. The formulas for these are given by 

\begin{figure}[H]
\centering \includegraphics[width=5in]{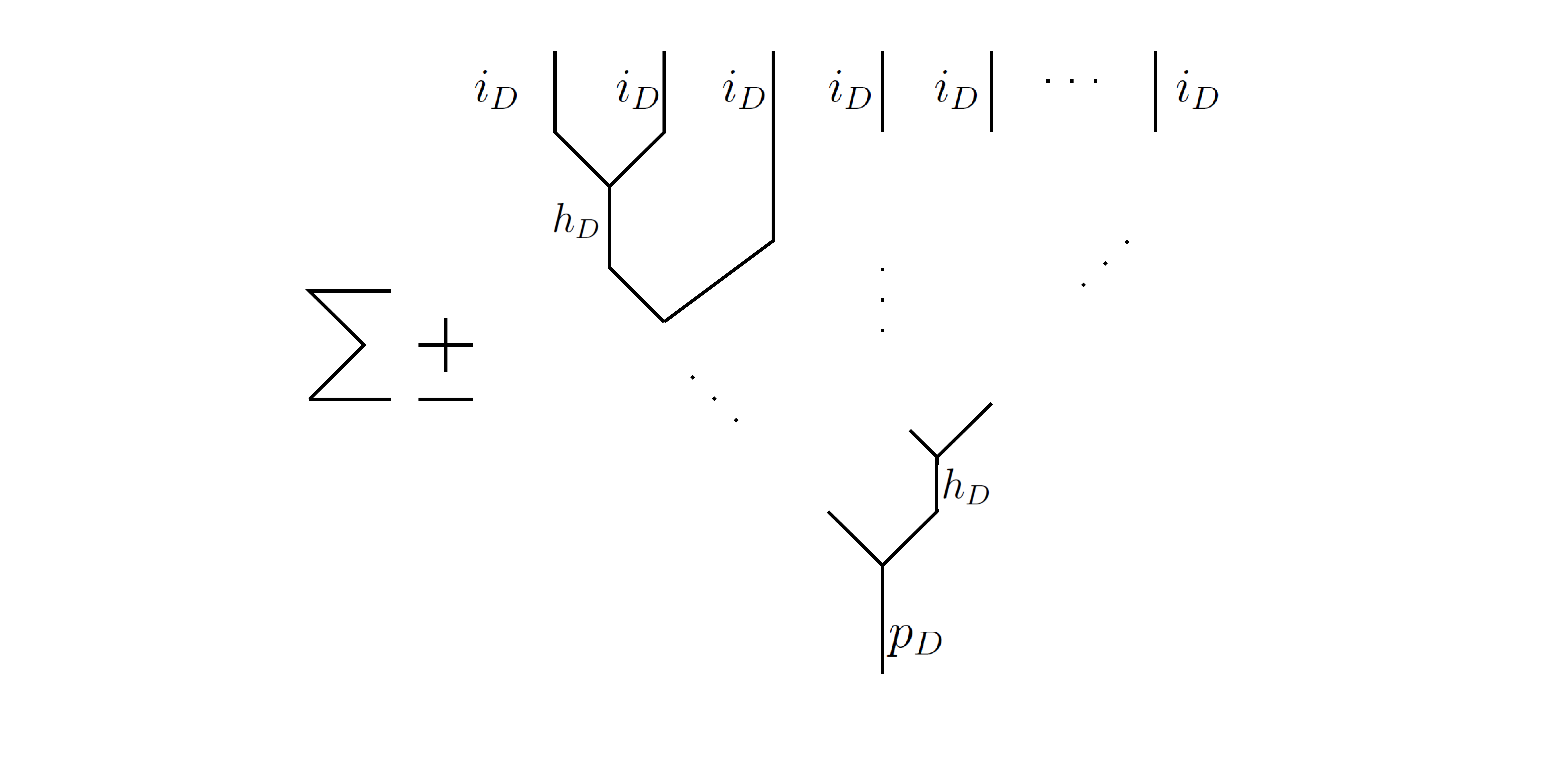}
\caption{}
\end{figure}
Similarly the formulas for the transferred structure on $C^*(M)$ from the differential forms is 

\begin{figure}[H]
\centering \includegraphics[width=5in]{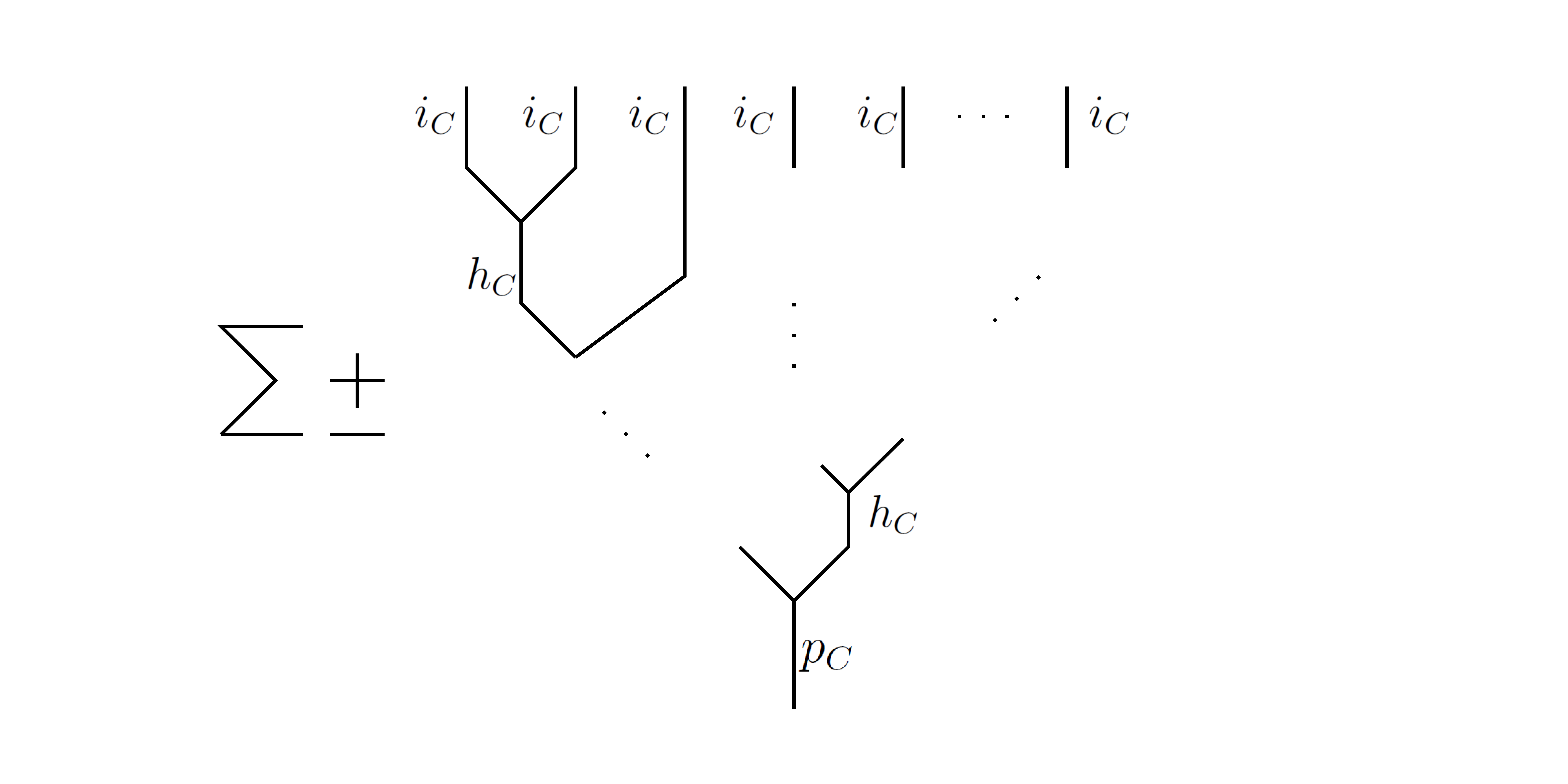}
\caption{}
\end{figure}

Since $h_C = i_D\circ h\circ p_D + h_D$ and $i_C = i\circ i_D$, and also since $p_D\circ i_D = id$ we get that the above sum is obtained by replacing the nodes in the first diagram by the trees in the second diagram.

\end{proof}

Thus given a finite tower of cochain complexes $C^*_1(M)$, $C^*_2(M)$, and so on, where $C^*_n(M)$ correspond to a finer cell decomposition than $C^*_{n-1}(M)$ we can transfer the graded commutative-associative structure from the differential forms to the tower in a compatible way. 

\section{$C_\infty$ morphism between \textit{dgca}s}

Suppose $A$ and $B$ are two \textit{dgca}s and $(p_1,p_2,\ldots)$ is a $C_\infty$ morphism from $A$ to $B$. Recall that the commutative cumulants are defined as follows.
\[k_n(a_1,a_2,\ldots,a_n) = \sum_\pi(|\pi|-1)!(-1)^{|\pi|-1}\prod_{b\in \pi}p(\prod_{i\in b}a_i)  \] Recall that $k_2$ is the same as the Boolean cumulant $K_2$ and thus from the first part of this paper \cite{Top.persp} it follows that $k_2$ is the boundary of the map $p_2$. 

\begin{lemma}
The commutative cumulants can be describes as boundaries of maps described using $p_2$
\end{lemma}
\begin{proof}
Note that the coefficients of $k_n$ are integers that must add up to zero. Also each term in $k_n$ corresponds to all partitions of $n$. Like in the previous section we associate a graph $G_n$ whose vertices correspond to the terms of $k_n$ with multiplicities. Edges go between a vertex $\alpha$ and $\beta$ if the partition corresponding to $\beta$ can be obtained from the partition corresponding to $\alpha$ by splitting one of its subsets into two. Note that since the coefficients of $k_n$ add up to zero, $G_n$ has even number of vertices. Also note that it is a connected graph. Any two terms corresponding to adjacent vertices in $G_n$ are homotopic to each other via $p_2$ and occur in $k_n$ with opposite signs. Thus we can take pairs of terms with opposite signs in $k_n$ that are homotopic to each other and use that to describe $k_n$ as a boundary.  

\end{proof}

The third cumulant $k_3$ is given by the formula 
\[k_3(a,b,c) = p(abc) - p(ab)p(c) - p(bc)p(a) - p(ca)p(b) + 2p(a)p(b)p(c)\] Thus the corresponding graph is
\begin{figure}[H]
\centering \includegraphics[width=5in]{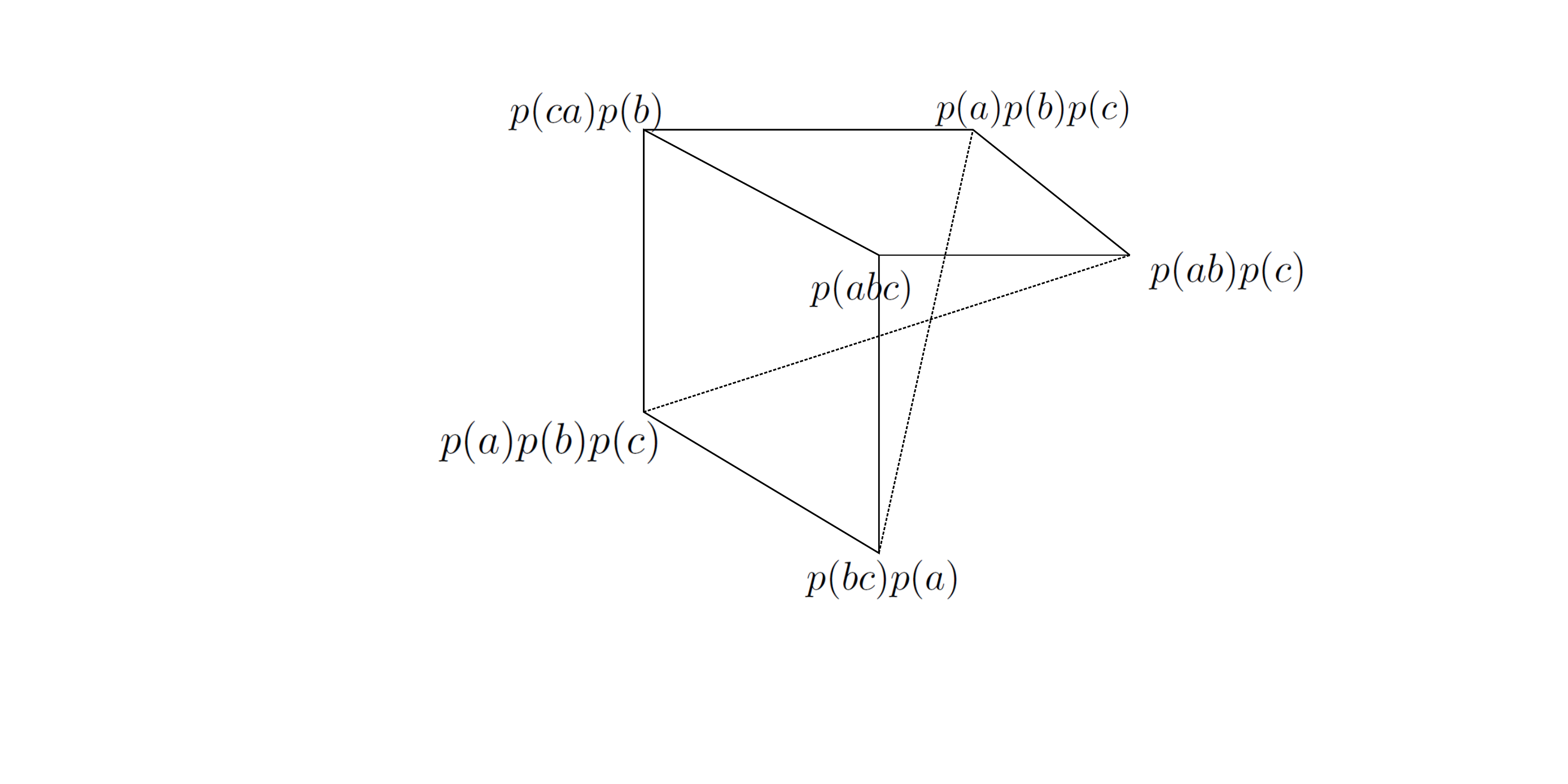}
\caption{}
\end{figure}

This is analogous to the discussion about making graphs associated to Boolean cumulants as described in the previous part of this paper \cite{Top.persp}. 

\begin{figure}[H]
\centering \includegraphics[width=5in]{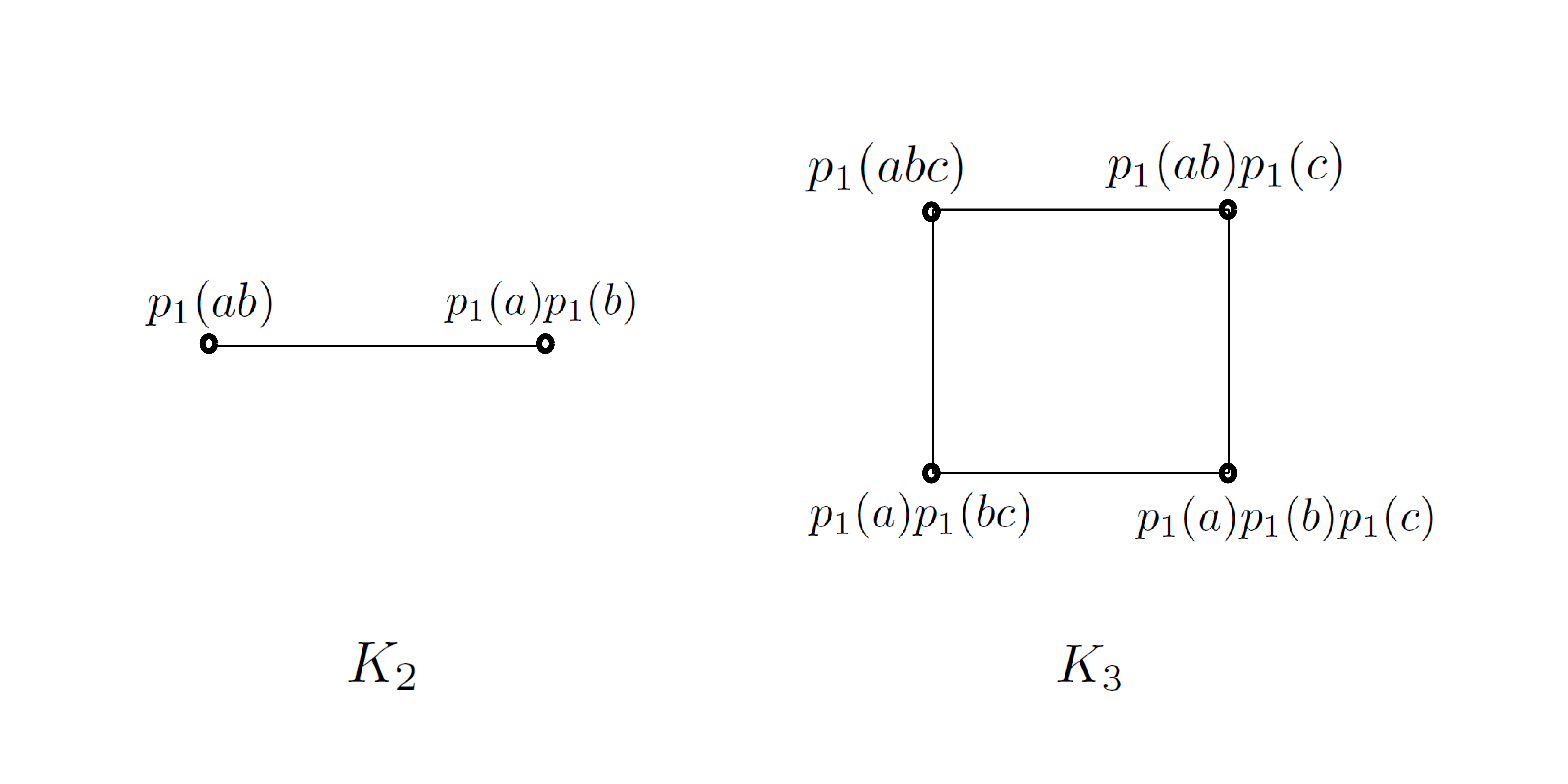}
\caption{}
\end{figure}

We further proved that such graphs can be made homotopic to zero using cells corresponding to maps that are part of the $A_\infty$ structure. In particular the two cells of such a complex correspond to

\begin{figure}[H]
\centering \includegraphics[width=5in]{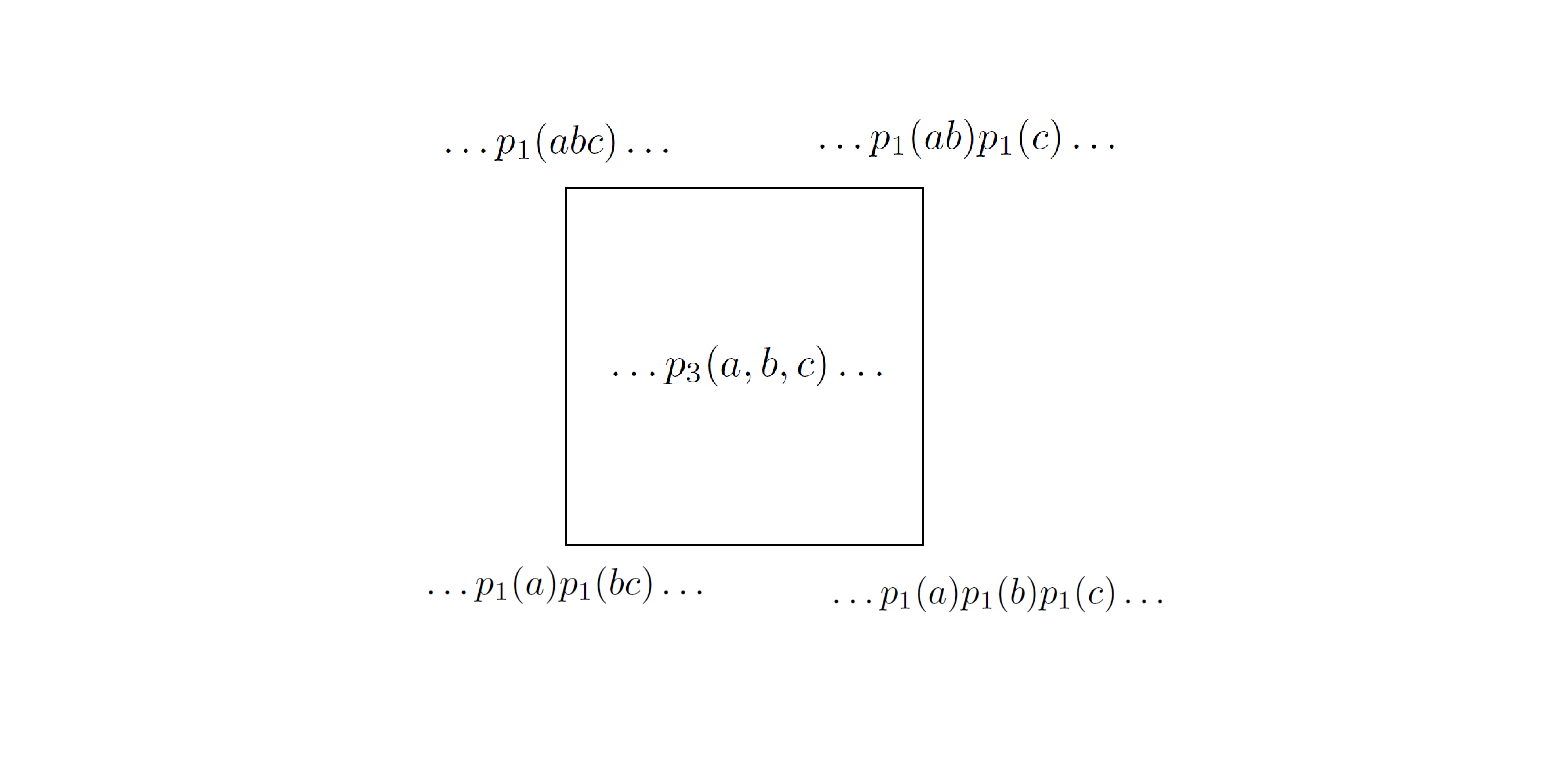}
\caption{\label{p3}}
\end{figure}

\begin{figure}[H]
\centering \includegraphics[width=5in]{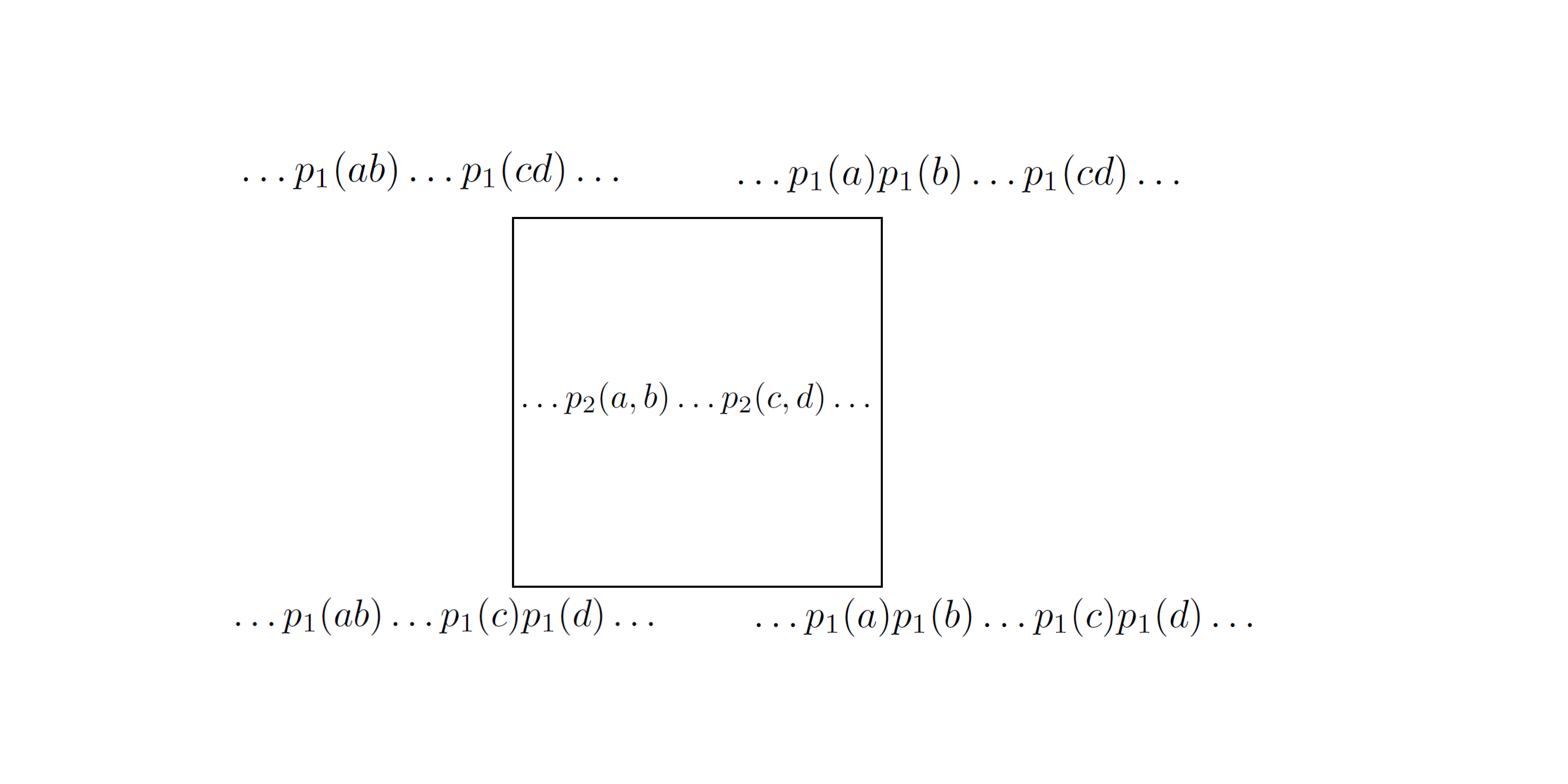}
\caption{\label{p2,p2}}
\end{figure}

We can now state the following theorem.

\begin{Theorem}
\label{thm:Comm}

Let $(A,d_A,\wedge_A)$ and $(B,d_B,\wedge_B)$ be two \textit{dgca}s. Let $p$ be a chain map from $A$ to $B$. Let $k_2$, $k_3$ and so on be the cumulants of $p$. Suppose $p$ is the first term of a $C_\infty$ morphism $(p, p_2, p_3, \ldots)$ where $p_n:A^{\otimes n}\rightarrow B$. Then the following statements hold.
\begin{itemize}
\item[i)] $p_2$ gives a homotopy from the second commutative cumulant $k_2$ to zero. All the higher cumulants $k_n$ are also homotopic to zero using maps created by $p_2$ and $p_1$.
\item[ii)] $p_3$ gives a homotopy between different ways of making $K_3$ homotopic to zero. For all the higher Boolean cumulants, homotopies between the multiple different ways of making them homotopic to zero are homotopic to each other using $p_3$, $p_2$ and $p_1$.
\item[iii)] In general any cycles that are created using the homotopies $\{p_j\}_{j=1}^n$ are made homotopic to zero using maps made by $\{p_j\}_{j=1}^{j+1}$.
\end{itemize}
\end{Theorem}

\begin{proof}
We will construct an $n-1$ dimensional cube complex $c_n$ corresponding to $k_n$ whose one skeleton is $G_n$. We first attach two cells corresponding to maps of the types that are described in figures \ref{p3} and \ref{p2,p2}. The boundaries of those maps correspond to $2$-cycles in $G_n$ since the edges in $G_n$. 
For every $j$ less then $n$ we attach a $j$-cube corresponding to maps of the form $p_1\ldots p_{k+1} \ldots p_1$ and $p_1\ldots p_{j_1}\ldots p_{j_2}\ldots p_1$ attached along the cells corresponding to their boundaries. Recall that 

\[\partial(p_n) = d(p_n)-\sum_{k=1}^np_n(1\otimes\ldots d\ldots\otimes 1)\]\[=  \sum_{k=1}^{n-1}p_{n-1}(1\otimes\ldots \wedge_A\ldots\otimes 1)-\sum_{n_1+n_2=n} \wedge_B(p_{n_1}\otimes p_{n_2})\] and

\[\partial(\wedge_B(p_{n_1}\otimes p_{n_2})) = \wedge_B(\partial(p_{n_1})\otimes \partial(p_{n_2}))\] Thus the boundaries of the cubes correspond to the sum of lower dimensional cubes. 

The complex $c_n$ is constructed similarly to the complex $g_n$ constructed in the previous section. Since the terms of $k_n$ include all permutations of the inputs, $c_n$ consists of $n-1$ cubes corresponding to $p_n$ with permuted inputs, glued together in a certain way. Thus for some subset of permutations of $j$ elements, we have cycles of the form \[\ldots p_j(\sum (a_{\sigma(1)},a_{\sigma(2)}\ldots a_{\sigma(j)})\ldots \] Recall that by the definition of a $C_\infty$ morphism $p_j$ vanishes over the sum of all shuffle permutations adding up to length $j$. Thus the map corresponding to the above sums is zero.

\end{proof}

\section{Structure of a general $C_\infty$ morphism}

Much like in the case of $A_\infty$ algebras, for maps between $C_\infty$ algebras there are multiple ways of defining the cumulants that are all homotopic to each other. The multiple ways are because of multiple ways of associating the variables in a product. Since these are homotopic in a $C_\infty$ algebra, the multiple ways of defining the $k_n$ are homotopic to each other.

\begin{Theorem}
\label{thm:C_infty}
Let $A$ and $B$ be two $C_\infty$ algebras. Let $p$ be a chain map from $A$ to $B$. Let $k_2$, $k_3$ and so on be the Boolean cumulants of $p$ defined up to homotopy. Suppose $p$ is the first term of an $C_\infty$ morphism $(p, p_2, p_3, \ldots)$ where $p_n:A^{\otimes n}\rightarrow B$. Then the following statements hold.
\begin{itemize}
\item[i)] $p_2$ gives a homotopy from the second cumulant $k_2$ to zero. All the different ways of defining the higher cumulants $k_n$ are also homotopic to zero using maps created by $p_2$ and $p_1$.
\item[ii)] $p_3$ gives a homotopy between different ways of making $k_3$ homotopic to zero. For all the higher Boolean cumulants, homotopies between the multiple different ways of making them homotopic to zero are homotopic to each other using $p_3$, $p_2$ and $p_1$.
\item[iii)] In general any cycles that are created using the homotopies $\{p_j\}_{j=1}^n$ are made homotopic to zero using maps made by $\{p_j\}_{j=1}^{n+1}$.
\end{itemize}

\end{Theorem}

\begin{proof}
Much like in the previous cases we construct a CW-complex for every $n$. In the case of a $C_\infty$ morphism between $C_\infty$ algebras the $n$th complex is made of the $n-1$ dimensional polyhedrons corresponding to the Boolean cumulants in the $A_\infty$ case. The cycles that aren't boundaries in this complex correspond to sums of $p_j$ and $m_j$ with permuted inputs. Recall that by the definition of $C_\infty$ algebras we have 

\[m_{q+r}(\sum_{\sigma \in (q,r)-\text{shuffles}} (x_{\sigma^{-1}(1)},x_{\sigma^{-1}(2)},\ldots,x_{\sigma^{-1}(q+r)}))=0\] and
\[p_{q+r}(\sum_{\sigma \in (q,r)-\text{shuffles}} (x_{\sigma^{-1}(1)},x_{\sigma^{-1}(2)},\ldots,x_{\sigma^{-1}(q+r)}))=0\] where $\mu$ is the shuffle product. Thus sums of cells corresponding to  $m_{q+r}$ and $p_{q+r}$ applied to shuffle products are cycles in the CW-complex. However these maps are also boundaries since they are indeed equal to zero. Thus we can add cells corresponding to the zero map whose boundaries are the above cycles. Thus the CW-complex is indeed contractible and the corresponding maps are boundaries.

\end{proof}

\nocite{*}

\bibliography{PaperII}
\bibliographystyle{plain}

\end{document}